\documentclass[11pt]{amsart}
\usepackage{amsmath, amsthm, amssymb, latexsym,url}
\usepackage{hyperref}

\allowdisplaybreaks

\author{Byron Heersink and Joseph Vandehey}
\thanks{
Email: \href{mailto:heersin2@illinois.edu}{\nolinkurl{heersin2@illinois.edu}}\\ University of Illinois at Urbana-Champaign\\
Email: \href{mailto:vandehey@uga.edu}{\nolinkurl{vandehey@uga.edu}}\\ University of Georgia}
\title[Continued fractions normality and arithmetic progressions]{Continued fraction normality is not preserved along arithmetic progressions}
\date{\today}

\newtheorem{thm}{Theorem}[section]

\newtheorem{lem}[thm]{Lemma}

\theoremstyle{remark}

\usepackage{bm}
\usepackage{graphicx}
\usepackage{epic}
\usepackage{eepic}
\usepackage{mathrsfs}
\usepackage{setspace}
\newcommand{\R}{\mathbb{R}}

\newcommand{\N}{\mathbb{N}}
\newcommand{\rar}{\rightarrow}

\begin{document}

\begin{abstract}
It is well known that if $0.a_1a_2a_3\dots$ is the base-$b$ expansion of a number normal to base-$b$, then the numbers $0.a_ka_{m+k}a_{2m+k}\dots$ for $m\ge 2$, $k\ge 1$ are all normal to base-$b$ as well.

In contrast, given a continued fraction expansion $\langle a_1,a_2,a_3,\dots\rangle$ that is normal (now with respect to the continued fraction expansion), we show that for any integers $m\ge 2$, $k\ge 1$, the continued fraction $\langle a_k, a_{m+k},a_{2m+k},a_{3m+k},\dots\rangle$ will never be normal.
\end{abstract}

\maketitle

\section{Introduction}

A number $x\in  [0,1)$ with base $10$ expansion $x=0.a_1a_2a_3\dots$ is said to be normal (to base $10$) if for any finite string $s=[c_1,c_2,\dots,c_k]$ of digits in $\{0,\ldots,9\}$, we have that 
\[
\lim_{n\to \infty} \frac{\#\{0\le i \le n: a_{i+j} = c_j, 1\le j \le k\}}{n} = \frac{1}{10^k}.
\]
Although almost all real numbers are normal, we still do not know of a single commonly used mathematical constant, such as $\pi$, $e$, or $\sqrt{2}$, that is normal.

A classical result due to Wall \cite{Wall} says that if $0.a_1a_2a_3\dots$ is normal, then so is $0.a_ka_{m+k}\\ a_{2m+k}a_{3m+k}\dots$, for any positive integers $k,m$. In concise terms, sampling along an arithmetic progression of digits preserves normality for base $10$ (and more generally, base $b$) expansions. Sampling along other sequences has been studied most notably by Agafonov \cite{Agafonov}, Kamae \cite{Kamae}, and Kamae and Weiss \cite{KW}. Merkle and Reimann \cite{MR} studied methods of sampling that do not preserve normality.

However, these works have focused primarily on base-$b$ expansions and so equivalent questions for other expansions are mostly unknown.

In this paper, we consider continued fraction expansions given by 
\[
x = \cfrac{1}{a_1+\cfrac{1}{a_2+\cfrac{1}{a_3+\dots}}} = \langle a_1,a_2,a_3,\dots\rangle, \quad a_i \in \mathbb{N}
\]
for $x\in [0,1)$.  The Gauss map $T$ is given by $Tx= x^{-1}-\lfloor x^{-1} \rfloor$ or, if $x=0$, then $Tx=0$, and it acts as a forward shift on continued fraction expansions, so that
\[
T\langle a_1,a_2,a_3,\dots\rangle = \langle a_2,a_3,a_4,\dots\rangle.
\] The Gauss measure $\mu$ on $[0,1)$ is given by
\[
\mu(A) = \int_A \frac{1}{(1+x)\log 2} \ dx.
\]
Given a finite string $s=[d_1,d_2,\dots,d_k]$ of positive integers we define the cylinder set $C_s$ to be the set of points $x\in [0,1)$ such that the string $[a_1,a_2,\dots,a_k]$ of the first $k$ digits of $x$ equals $s$. (The expansions of rational numbers are finite and non-unique, but we may ignore such points throughout this paper.)

We say that $x\in [0,1)$ is CF-normal if, for any finite, non-empty string $s=[d_1,d_2,\dots,d_k]$ of positive integers, we have
\[
\lim_{n\to \infty} \frac{\#\{0\le i \le n : T^i x \in C_s\}}{n} = \mu(C_s),
\]
which is equivalent to saying that the limiting frequency of $s$ in the expansion of $x$ equals $\mu(C_s)$, since $T^i x \in C_s$ if and only if the string $[a_{i+1},a_{i+2},\dots,a_{i+k}]$ equals $s$. By the ergodicity of the Gauss map $T$ and the pointwise ergodic theorem, almost all $x\in [0,1)$ are CF-normal.

\begin{thm}\label{thm:main}
Suppose $\langle a_1, a_2,a_3,\dots\rangle$ is CF-normal. Then the number $\langle a_k, a_{m+k}, a_{2m+k},\\ a_{3m+k}, \dots\rangle$ is not CF-normal for any integers $k\ge 1$, $m\ge 2$. In fact, for any integers $k\ge 1$, $m\ge 2$, we have that
\[
\lim_{n\to \infty} \frac{\#\{1\le i \le n: a_{(i-1)m+k} =a_{im+k}=1\}}{n} 
\]
exists, but does not equal $\mu(C_{[1,1]})$, so that the CF-normality of $\langle a_k, a_{m+k}, a_{2m+k}, a_{3m+k}, \dots\rangle$  can be seen to fail just by examining the frequency of the string $[1,1]$.
\end{thm}

One of the key techniques in proving this result is a way of augmenting the usual Gauss map $T$ to simultaneously act on a finite-state automata. A number of recent results have made use of this blending of ergodicity and automata. It was used in Agafonov's earlier cited result \cite{Agafonov}. It was used in Jager and Liardet's proof of Moeckel's theorem (where it was called a skew product) \cite{JL}.  It was used to study normality from the viewpoint of compressability \cite{BCH,BH}. And it was used by Blanchard, Dumont, and Thomas to give reproofs of some classical normality equivalencies, even extending some of these results to what they call ``near-normal" numbers \cite{BDT,Blanchard}.
 
We end the introduction with two questions.

First, the proof of Theorem \ref{thm:main} could be extended to show that any continued fraction expansion formed by selecting along a non-trivial arithmetic progression of digits from a CF-normal number has all its $1$-digit strings appearing with the right frequency, but the 2-digit string $[1,1]$ does not. We wonder whether any string with more than one digit can appear with the correct frequency for CF-normality, or are they always incorrect.

Second, as stated earlier, sampling along a non-trivial arithmetic progression preserves normality for base-$b$ expansions. It can be shown, using, say, the augmented systems in this paper, that a similar result holds for any fibred system that is Bernoulli. The continued fraction expansion is a simple example of a non-Bernoulli system. Is Bernoullicity not only sufficient but necessary for selection along non-trivial arithmetic progressions to preserve normality?

\section{An augmented system}

We will require a result from a previous paper of the second author \cite{ratmultCF}.

Let $T$ be the Gauss map acting on the set $\Omega \subset[0,1)$ of irrationals. So $Tx \equiv 1/ x\pmod{1}$.  We will consider cylinder sets of $\Omega$ to be the intersection of the usual cylinder sets (for the continued fraction expansion) of $[0,1)$ with $\Omega$. 

We wish to extend the map $T$ to a transformation $\widetilde{T}$ on a larger domain $\widetilde{\Omega}=\Omega \times \mathcal{M}$ for some finite set $\mathcal{M}$. For any $(x,M)\in \widetilde{\Omega}$, we define \[\widetilde{T}(x,M) = (Tx,f_{a_1(x)} (M)),\] where $a_1(x)=\lfloor x^{-1}\rfloor$ is the first continued fraction digit of $x$ and the functions $f_a: \mathcal{M}\to\mathcal{M}$, $a\in \mathbb{N}$, are bijective.  Since the second coordinate of $\widetilde{T}(x, M)$ only depends on $M$ and the first digit of $x$, we see that this second coordinate is constant for all $x$ in the same rank $1$ cylinder. Given a cylinder set $C_s$ for $\Omega$, we call $C_s \times \{M\}$ (for any $M\in \mathcal{M}$) a cylinder set for $\widetilde{\Omega}$. We also have a measure $\tilde{\mu}$ on $\widetilde{\Omega}$ that is defined as being the product of the Gauss measure on $\Omega$ times the counting measure on $\mathcal{M}$, normalized by $1/|\mathcal{M}|$ to be a probability measure. By the assumed bijectivity of $f$, we have that $\widetilde{T}$ preserves $\tilde{\mu}$.

For easier readability, we will use $(E,M)$ to denote $E \times \{M\}$ for any measurable set $E\subset \Omega$, with measurability being determined by Lebesgue measure or, equivalently, the Gauss measure. 

We adapt our definition of normality on this space. We will say that $(x,M)\in \widetilde{\Omega}$ is $\widetilde{T}$-normal with respect to $\tilde{\mu}$, if for any cylinder set $(C_s,M')$ we have
\[
\lim_{n\to \infty} \frac{\#\{0\le i < n: \widetilde{T}^i(x,M)\in (C_s,M')\}}{n} = \tilde{\mu} (C_s,M').
\]

We say $\widetilde{T}$ is transitive if for any $M_1,M_2\in \mathcal{M}$, there exists a proper string $s$ of length $n$ such that \[ T^n( C_s,M_1) = (\Omega,M_2). \]

\begin{thm}\label{thm:traversing}
If $\widetilde{T}$ is transitive, then  $\widetilde{T}$ is ergodic with respect to $\tilde{\mu}$.  Moreover, if $x$ is normal, then for any $M\in \mathcal{M}$, the point $(x, M)$ is $\widetilde{T}$-normal with respect to $\tilde{\mu}$.
\end{thm}

In \cite{ratmultCF}, this result was proved without assuming the bijectivity of the functions $f_a$. This results in being unable to assume that $\tilde{\mu}$ is an invariant measure and makes the overall proof significantly more difficult.

\section{An operator-analytic lemma}

Let $A=C_{[1]}=[1/2,1)$. It can be easily calculated that
\[
\mu(C_{[1]})=\mu(A) = \frac{\log(4/3)}{\log 2} \quad\text{and} \quad \mu(C_{[1,1]}) = \mu(T^{-1} A\cap A) = \frac{\log (10/9)}{\log 2}.
\]
Moreover, since $T$ is known to be strong mixing, we have that
\[
\lim_{n\to \infty}\mu(T^{-n} A \cap A)  = \mu(A)^2 = \left( \frac{\log (4/3)}{\log 2}\right)^2.
\]

\begin{lem}\label{lem:Wirsing}
We have
\begin{equation}\label{eq:main}
\mu( T^{-n} A\cap A) < \mu( T^{-1} A\cap A)
\end{equation}
for any integer $n\ge 2$.
\end{lem}

\begin{proof}
%Let $A=C_{[1]}=[1/2,1]$. We want to prove that
%\begin{equation}\label{main}
%\frac{\mu(A\cap T^{-n}A)}{\mu(A)} \neq \frac{\mu(A\cap T^{-1}A)}{\mu(A)}=\frac{\log(10/9)}{\log(4/3)}.\quad(n\geq2)
%\end{equation}
We closely follow a process of Wirsing \cite{W} which established the spectral gap in the transfer operator of $T$, and in turn gave a very precise estimate of
\[\left|\mu(T^{-n}[0,x))-\mu([0,x))\right|\]
as $n\to\infty$. Through this process, we prove the bound
\[\left|\frac{\mu(A\cap T^{-n}A)}{\mu(A)} - \mu(A)\right| < \mu(A)-\frac{\log(10/9)}{\log(4/3)}=\frac{\log(4/3)}{\log2}-\frac{\log(10/9)}{\log(4/3)},\qquad(n\geq2)\]
which implies \eqref{eq:main}.

To start with, define $m_n,r_n:[0,1]\to\R$ by
\[m_n(x)=\frac{\mu(A\cap T^{-n}[0,x))}{\mu(A)}\quad\mbox{and}\quad r_n(x)=m_n(x)-\mu([0,x)).\]
Then
\begin{align*}
r_n\left(\frac{1}{2}\right)=\frac{\mu(A\cap T^{-n}[0,1/2))}{\mu(A)}-1+1-\mu([0,1/2))=\mu(A)-\frac{\mu(A\cap T^{-n}A)}{\mu(A)},
\end{align*}
and so we want to bound $|r_n(1/2)|$. Next, we introduce the transfer operator of $T$, which is the map $\hat{T}:L^1(\mu)\rar L^1(\mu)$ satisfying
\[\int_{B}\hat{T}f\,d\mu=\int_{T^{-1}(B)}f\,d\mu\mbox{, for all Borel subsets }B\subseteq[0,1)\mbox{ and }f\in L^1(\mu),\]
and is given by the formula
\begin{equation}\label{transfereq}
(\hat{T}f)(x)=\sum_{k=1}^\infty\frac{1+x}{(k+x)(k+1+x)}f\left(\frac{1}{k+x}\right), \quad x\in (0,1).
\end{equation}
This formula may be extended in the natural way to functions on $[0,1]$. When extended, $\hat{T}$ is also an operator from $C^1[0,1]$ to itself. Moreover, if $f=g$ Lebesgue-a.e. then $\hat{T}f=\hat{T}g$ Lebesgue-a.e. We have
\[m_n(x)=\frac{1}{\mu(A)}\int_0^x(\hat{T}^n1_A)(t)\,d\mu(t)=\frac{1}{\mu(A)\log2}\int_0^x(\hat{T}^n1_A)(t)\,\frac{dt}{1+t},\]
where $1_A$ is the indicator function of $A$. Therefore, $m_n'$ exists Lebesgue-a.e.~and
\[(1+x)m_n'(x)=\frac{1}{\mu(A)\log2}(\hat{T}^n1_A)(x)\quad\mbox{for Lebesgue-a.e.~}x.\]
Now by \eqref{transfereq}, we clearly have
\[(\hat{T}1_A)(x)=
\frac{1}{2+x}
\]
if $x\in (0,1)$.
So if we define $f_1(x)=\frac{1}{(2+x)\mu(A)\log2}$ and $f_n=\hat{T}^{n-1}f_1$, then $f_n=\frac{1}{\mu(A)\log 2}\hat{T}^n1_A$ Lebesgue-a.e. Since $\hat{T}$ preserves continuity on $[0,1]$, each $f_n$ is continuous, so we can say that $m_n'$ exists on all of $[0,1]$, and $f_n(x)=(1+x)m_n'(x)$ for all $x\in[0,1]$ and $n\in\N$.

Next, we define $g_n(x)=f_n'(x)$, noting that $f_n\in C^1[0,1]$ for all $n\in\N$. We then have $g_{n+1}(x)=-(Ug_n)(x)$, where $U$ is the operator examined by Wirsing, defined by $U(f')=-(\hat{T}f)'$, and can be shown to be given by
\[(Ug)(x)=\sum_{k=1}^\infty\left(\frac{k}{(k+1+x)^2}\int_{1/(k+1+x)}^{1/(k+x)}g(y)\,dy+\frac{1+x}{(k+x)^3(k+1+x)}g\left(\frac{1}{k+x}\right)\right).
\]
The operator $U$ is clearly positive so that $Ug\leq Uf$ whenever $g\leq f$.

%(The following is how to prove this equality, and hence that $U$ is positive, which probably should be skipped in the paper (this is essentially copied from Wirsing's paper): Let $f$ and $g$ be two continuous functions such that $f'=g$. Then
%\begin{align*}
%(Ug)(x)&=-\frac{d}{dx}(\hat{T}f)(x)=-\frac{d}{dx}\sum_{k=1}^\infty\left(\frac{k}{k+1+x}-\frac{k-1}{k+x}\right)f\left(\frac{1}{k+x}\right)\\
%&=-\sum_{k=1}^\infty\left(\frac{k-1}{(k+x)^2}-\frac{k}{(k+1+x)^2}\right)f\left(\frac{1}{k+x}\right)+\sum_{k=1}^\infty\frac{1+x}{(k+x)^3(k+1+x)}g\left(\frac{1}{k+x}\right)\\
%&=\sum_{k=1}^\infty\frac{k}{(k+1+x)^2}\left(f\left(\frac{1}{k+x}\right)-f\left(\frac{1}{k+1+x}\right)\right)+\sum_{k=1}^\infty\frac{1+x}{(k+x)^3(k+1+x)}g\left(\frac{1}{k+x}\right)\\
%&=\sum_{k=1}^\infty\left(\frac{k}{(k+1+x)^2}\int_{1/(k+1+x)}^{1/(k+x)}g(y)\,dy+\frac{1+x}{(k+x)^3(k+1+x)}g\left(\frac{1}{k+x}\right)\right),
%\end{align*}
%which is clearly nonnegative if $g\geq0$.)

We have $f_1(x)=1/((x+2)\log(4/3))$, and so $g_1(x)=-1/((x+2)^2\log(4/3))$. From the work of Wirsing, $U(-g_1)\leq-\frac{1}{2}g_1$. This can be shown as follows.
Let $a(x)=1/(x+2)^2$, $b(x)=1/(1+2x)^2$, and $c(x)=-1/(2+4x)$ so that $a\leq b$ on $[0,1]$ and $c'=b$. For $x\in[0,1]$, we have
\begin{align*}
(Ua)(x)&\leq(Ub)(x)=-(\hat{T}c)'(x)=\frac{d}{dx}\sum_{k=1}^\infty\frac{1+x}{(k+x)(k+1+x)}\frac{1}{2+4/(k+x)}\\
&=\frac{1}{2}\frac{d}{dx}\sum_{k=1}^\infty\frac{1+x}{(k+1+x)(k+2+x)}=\frac{1}{2}\frac{d}{dx}\sum_{k=1}^\infty\left(\frac{1+x}{k+1+x} - \frac{1+x}{k+2+x}\right)\\
&=\frac{1}{2}\frac{d}{dx}\left(\frac{1+x}{2+x}\right)=\frac{1}{2(2+x)^2}=\frac{1}{2}a(x),
\end{align*}
implying that $g_2 = U(-g_1)\leq-\frac{1}{2}g_1$, and hence, by iterating this procedure and recalling that $g_{n+1}=-Ug_n$, we get that $|g_n|\leq-\frac{1}{2^{n-1}}g_1$.

Now let $\xi=\log(1+x)$ and $\varrho_n(\xi)=r_n(x)$. Then note that
\begin{align*}
\varrho_n''(\xi)&=\frac{d^2}{d\xi^2}r_n(e^\xi-1)=\frac{d}{d\xi}(e^\xi r_n'(e^\xi-1))=e^\xi r_n'(e^\xi-1)+e^{2\xi}r_n''(e^\xi-1)\\
&=(1+x)(r_n'(x)+(1+x)r_n''(x))\\
&=(1+x)\left(m_n'(x)-\frac{1}{(1+x)\log2}+(1+x)\left(m_n''(x)+\frac{1}{(1+x)^2\log2}\right)\right)\\
&=(1+x)(m_n'(x)+(1+x)m_n''(x))=(1+x)\frac{d}{dx}((1+x)m_n'(x))=(1+x)g_n(x).
\end{align*}
We have $r_n(0)=r_n(1)=0$, $\varrho_n(0)=\varrho_n(\log2)=0$, and so by the mean value theorem of divided differences, 
\[\varrho_n(\xi)=-\xi(\log2-\xi)\frac{\varrho_n''(\xi^*)}{2}\]
for some $\xi^*\in[0,\log2]$ depending on $\xi$. Letting $\xi=\log(3/2)$ and taking absolute values yields
\begin{align*}
\left|r_n\left(\frac{1}{2}\right)\right|&\leq\frac{1}{2}\left(\log\frac{3}{2}\right)\left(\log2-\log\frac{3}{2}\right)\|\varrho_n''\|_\infty=\frac{1}{2}\left(\log\frac{3}{2}\right)\left(\log\frac{4}{3}\right)\|(1+x)g_n(x)\|_\infty\\
&=\frac{1}{2^n }\left(\log\frac{3}{2}\right)\left(\log\frac{4}{3}\right)\left\|(1+x)g_1(x)\right\|_\infty\leq\frac{1}{2^n}\log\frac{3}{2}\left\|\frac{1+x}{(x+2)^2}\right\|_\infty=\frac{1}{2^{n+2}}\log\frac{3}{2}.
\end{align*}
If $n\geq2$, this is at most $\frac{1}{16}\log\frac{3}{2}=0.025341\ldots$, which is less than $\frac{\log(4/3)}{\log2}-\frac{\log(10/9)}{\log(4/3)}=0.048798\ldots$ This completes the proof of Lemma \ref{lem:Wirsing}.

\end{proof}

\section{Proof of Theorem \ref{thm:main}}

Without loss of generality, it suffices to prove the theorem if $1 \le k \le m$.

Consider the augmented system $\widetilde{T}$ on $\widetilde{\Omega}$ given by $\mathcal{M} = \{1,2,\dots,m\}$ and $f_a(k) = k+1 \pmod{m}$ for all $a\in \mathbb{N}$. In particular, we always have that
\[
\widetilde{T}^i (x,j) = (T^i x, j+i\bmod{m}).
\]
Also, it is clear that this is transitive: for any rank $n$ cylinder, we have that $\widetilde{T}^n (C_s, j_1) = (\Omega, j_1+n \pmod{m})$. Therefore Theorem \ref{thm:traversing} applies.

Let $x=[a_1,a_2,a_3,\dots]$ be CF-normal, and let $y=[a_k,a_{m+k},a_{2m+k},\dots]$. Consider the string $s=[1,1]$. We want to show that the limiting frequency of $s$ in the digits of $y$ does not equal $\mu(C_s)$. 

Borrowing our notation from the last section, we let $A=C_{[1]}$ and we will now denote $A\cap T^{-n} A$ by $E_n$, so that $C_s = E_1$.

We have that $T^i y \in E_1$ if and only if $T^{mi+k-1} x \in E_m$. Note that $(x,1)$ is normal with respect to $\widetilde{T}$ by Theorem \ref{thm:traversing}. Thus we have that
\begin{align*}
\lim_{n\to \infty}  \frac{\#\{0\le i \le n : T^i y \in C_s\}}{n} &= \lim_{n\to \infty}  \frac{\#\{0\le i \le n : T^{mi+k-1} x \in E_m\}}{n} \\
&= \lim_{n\to \infty} \frac{\#\{ 0 \le i \le mn: \widetilde{T}^i (x,1) \in (E_m, k)\}}{n}\\
&= m\cdot \lim_{n\to \infty} \frac{\#\{ 0 \le i \le mn: \widetilde{T}^i (x,1) \in (E_m, k)\}}{mn}\\
&= m \cdot \tilde{\mu}(E_m,k)  = m\cdot \frac{\mu(E_m)}{m} = \mu(E_m).
\end{align*}

By Lemma \ref{lem:Wirsing}, we have that $\mu(E_m) < \mu(C_s)$, which proves the theorem.

\section{Acknowledgments}

The authors would like to thank Florin Boca for his suggestions.

The research of Joseph Vandehey was supported in part by the NSF grant DMS-1344994 of the RTG in Algebra, Algebraic Geometry, and Number Theory, at the University of Georgia.

\end{document}